\documentclass[12pt,a4paper,reqno]{amsart}
\usepackage[centering, margin=1in]{geometry}
\usepackage{amsmath,amsfonts,amsthm,amssymb,enumitem,mathrsfs,mathtools}
\usepackage[pagebackref,linkcolor=blue,citecolor=red,filecolor=blue,
urlcolor=blue,unicode]{hyperref}
\usepackage{xcolor}

\usepackage[hang,flushmargin]{footmisc} 

\hypersetup{pagebackref,
		pdfnewwindow=true,
	colorlinks, linkcolor={red},
	citecolor={red}, urlcolor={blue}
}

\usepackage{tabto}

\patchcmd{\thebibliography}
{\settowidth}
{\setlength{\parsep}{0pt}\setlength{\itemsep}{10pt plus 0.1pt}\settowidth}
{}{}

\DeclareMathOperator{\li}{li}
\DeclareMathOperator{\ord}{ord}

\newtheorem{thm}{Theorem}[section]
\newtheorem{lem}{Lemma}[section]

\newtheorem{dfn}{Definition}[section]

\newcommand{\N}{\mathbb{N}}
\newcommand{\Z}{\mathbb{Z}}

\newcommand{\R}{\mathbb{R}}
\newcommand{\C}{\mathbb{C}}
\newcommand{\F}{\mathbb{F}}
\newcommand{\tP}{\mathbb{P}}

\newcommand{\T}{\mathbb{T}}

\usepackage{fancyhdr}
\begin{document}

\title{Density of Squarefree Totients $p-1$ and Primitive Roots}


\author{N. A. Carella}
\address{}
\curraddr{}
\email{}
\thanks{}




\date{\today}

\dedicatory{}

\begin{abstract}This note determines an effective asymptotic formula for the number of squarefree totients $p-1$ with a fixed primitive root $u\ne \pm 1, v^2$.
\end{abstract}

\maketitle
\tableofcontents
\pagenumbering{Page gobble}
\pagenumbering{arabic}
\section{\textbf{Introduction}}\label{S0987SFT}\hypertarget{S0987SFT}
Let $x\geq 1$ be a large number, and fix an integer $u\ne\pm1,v^2$. The problem of the determining the density of a subset of primes $p\leq x$ with a fixed primitive root, better known as Artin primitive root conjecture, has a large literature. The conditional proof is derived in \cite{HC1967}. Many other related results for primitive roots appears in \cite{GM1984}, \cite{HB1986}, \cite{MP2004}, \cite{LH1977}, et alii. This note is concerned with determining the density of a subset of primes $p\leq x$ with squarefree totients $p-1$ and a fixed primitive root $u$. More precisely, the subset of such primes is defined by
\begin{equation}\label{eq0987SFT.100d}
	\mathscr{P}_{sf}(u) = \{p \in\tP : \mu(p-1)\ne0 \text{ and } \ord_p(u) = p-1\}.
\end{equation}
The corresponding counting function is defined by
\begin{equation}\label{eq0987SFT.100f}
\pi_{sf}(x,u) = \#\{p \leq x : \mu(p-1)\ne 0 \text{ and } \ord_p(u) = p-1\}.
\end{equation}
There is no literature on the counting function \eqref{eq0987SFT.100f}. The first  asymptotic formula for this function is derived here.
\begin{thm}\label{thm0987SFT.100}\hypertarget{thm0987SFT.100}
Let $x\geq 1$ be a large number and let $u\ne\pm1,v^2$
be a fixed integer. Then,
\begin{equation}\label{eq0987SFT.100j}
	\pi_{sf}(x,u) = \delta(u)\li(x)+O\left(\frac{x}{(\log x)^c} \right) ,
\end{equation}
where the density $\delta(u)$ of squarefree totients depends on $u$ and $c > 1$ is an arbitrary constant.
\end{thm}
After the preliminary preparation in \hyperlink{S9955D}{Section} \ref{S9955D} to \hyperlink{S5757SFT-E}{Section} \ref{S5757SFT-E}, the proof of \hyperlink{thm0987SFT.100}{Theorem} \ref{thm0987SFT.100} is completed in
\hyperlink{S5757SFT}{Section} \ref{S5757SFT}.

\section{\textbf{Representations of the Characteristic Function}} \label{S9955D}\hypertarget{S9955D}
The \textit{multiplicative order} of an element in a finite field is defined by $\ord_p u=\min\{k:u^k\equiv 1 \bmod p\}$. An element $u\ne \pm1,v^2$ is called a \textit{primitive root} if $\ord_p u=p-1.$ The \textit{characteristic function} \(\Psi :G\longrightarrow \{ 0, 1 \}\) of primitive elements is one of the standard analytic tools employed to investigate the various properties of primitive roots in cyclic groups \(G\). Many equivalent representations of the characteristic function $\Psi $ of primitive elements
are possible. Two of these representations are investigated here. 
\subsection{Divisor Dependent Characteristic Function}		
The \textit{divisor dependent} characteristic function was developed about a century ago, see {\color{red}\cite[Theorem 496]{LE1927}}, {\color{red}\cite[p.\; 258]{LN1997}}, et alia. This characteristic function detects the multiplicative order of an element by means of the divisors of the totient $p-1$. The precise description is stated below.

\begin{lem} \label{lem9955.200D} \hypertarget{lem9955.200D} Let \(p\geq 2\) be a prime and let \(\chi\) be a multiplicative character of multiplicative order $\ord  \chi =d$. If \(u\in
	\mathbb{F}_p\) is a nonzero element, then
	\begin{equation}
		\Psi (u)=\frac{\varphi(p-1)}{p-1}\sum _{d\mid p-1} \frac{\mu(d)}{\varphi(q)}\sum _{\ord \chi =d} \chi(u)
		=\left \{
		\begin{array}{ll}
			1 & \text{ if } \ord_p (u)=p-1,  \\
			0 & \text{ if } \ord_p (u)\neq p-1, \\
		\end{array} \right .\nonumber
	\end{equation}
	where $\mu:\N\longrightarrow \{-1,0,1\}$ is the Mobius function.
\end{lem}	
There are a few variant proofs of this result, these are widely available in the literature. Almost every result in the theory of primitive roots in finite fields is based on this characteristic function, but sometimes written in different forms. An extension of this characteristic function to the finite ring $\Z/n\Z$ is presented in {\color{red}\cite[Lemma 4]{MG1998}}. 
\subsection{Divisorfree Characteristic Function}	
A new \textit{divisors-free} representation of the characteristic function of primitive element is developed here. This is the original source of this result, it is the product of many years of research in number theory and finite fields. It detects the multiplicative order \(\text{ord}_p(u) \geq 1\) of the element \(u\in \mathbb{F}_p\) by means of the solutions of the equation \(\tau ^n-u=0\) in \(\mathbb{F}_p\), where
\(u,\tau\) are constants, and $n\in \mathscr{R}=\{n<p:\gcd(n,p-1)=1\}$ is a variable. 

\begin{lem} \label{lem9955.200A} \hypertarget{lem9955.200A} Let \(p\geq 2\) be a prime and let \(\tau\) be a primitive root mod \(p\) and  let \(\psi \neq 1\) be a nonprincipal additive character of order $\ord  \psi =p$. If \(u\in
	\mathbb{F}_p\) is a nonzero element, then
	\begin{equation}
		\Psi (u)=\sum _{\gcd (n,p-1)=1} \frac{1}{p}\sum _{0\leq s\leq p-1} \psi \left ((\tau ^n-u)s\right)
		=\left \{
		\begin{array}{ll}
			1 & \text{ if } \ord_p (u)=p-1,  \\
			0 & \text{ if } \ord_p (u)\neq p-1. \\
		\end{array} \right .\nonumber
	\end{equation}
\end{lem}	

\begin{proof}[\textbf{Proof}] Set the additive character \(\psi(s) =e^{i 2\pi  as/p}\). As the index $n\in \mathscr{R}=\{n<p:\gcd(n,p-1)=1\}$ ranges over the integers relatively prime to $\varphi(p)=p-1$, the element $\tau ^n\in \F_p ^{\times}$ ranges over the primitive roots
	modulo $p$. Accordingly, the equation $a=\tau ^n- u=0$ has a unique solution $n\geq1$ if and only if the fixed element $u\in \F_p$ is a primitive root. This implies that the inner sum in 	
	\begin{equation}\label{eq9977FR.300DF}
		\sum_{\gcd (n,p-1)=1} \frac{1}{p}\sum _{0\leq s< p} e^{i 2\pi \frac{(\tau ^n-u)s}{p}}=
		\left \{\begin{array}{ll}
			1 & \text{ if } \ord_{p} (u)=p-1,  \\
			0 & \text{ if } \ord_{p} (u)\ne p-1. \\
		\end{array} \right.
	\end{equation} 
	collapses to $\sum _{0\leq s< p} e^{i 2\pi as/p}=\sum _{0\leq s< p} 1=p $. Otherwise, if the element $u\in \F_p$ is not a primitive root, then the equation $a=\tau ^n- u=0$ has no solution $n\geq1$, and the inner sum in \eqref{eq9977FR.300DF} collapses to $\sum _{0\leq s< p} e^{i 2\pi as/p}=0$,
	this follows from the geometric series formula $\sum_{0\leq n\leq  N-1} w^n =(w^N-1)/(w-1)$, where $w=e^{i 2\pi a/p}\ne1$ and $N=p$. 
\end{proof}

\section{\textbf{Finite Fourier Transform and Summation Kernels}}

\subsection{Finite Fourier Transform}
Let $f: \C \longrightarrow \C$ be a function, and let $q \in \N$ be a large integer. 

\begin{dfn}\label{dfn4400FFT.300}{\normalfont 
The discrete Fourier transform of the function $f:\N\longrightarrow \C$ and its inverse are defined by
\begin{equation} \label{eq4400FFT.300d}
	\hat{f}(t)=\sum_{0 \leq s\leq q-1} e^{i \pi st/q}
\end{equation}
		and 
		\begin{equation}\label{4400FFT.300c}
			f(s)=\frac{1}{q}\sum_{0 \leq t\leq q-1}\hat{f}(m)e^{-i2\pi st/q},
		\end{equation}
		respectively.		
	}
\end{dfn} 

The finite Fourier transform and its inverse are used here to derive a summation kernel function, which is almost identical to the Dirichlet kernel, in this application $q=p$ is a prime number.

\begin{dfn} \label{dfn4400FFT.100} \hypertarget{dfn4400FFT.100}{\normalfont Let $ p$ be a prime, let $\omega=e^{i 2 \pi/p}$, and $\zeta=e^{i 2 \pi/p}$ be roots of unity. The \textit{finite summation kernel} is defined by the finite Fourier transform identity
		\begin{equation} \label{eq4400FFT.100g}
			\mathcal{K}(f(n))=\frac{1}{p} \sum_{0 \leq t\leq p-1,}  \sum_{0 \leq s\leq p-1} \omega^{t(n-s)}f(s)=f(n).\end{equation}
	} 
\end{dfn}
This simple identity is very effective in computing upper bounds of some exponential sums
\begin{equation} \label{eq4400FFT.100h}
	\sum_{ n \leq x}  f(n)= \sum_{ n \leq x}  \mathcal{K}(f(n)),
\end{equation}
where $x  < p$. 

\subsection{Summation Kernels} 
\begin{lem}   \label{lem4400SK.150A}\hypertarget{lem4400SK.150A}  Let \(p\geq 2\)  be a large primes, let $x<p-1$ and let $\omega=e^{i2 \pi/p} $ be a $p$th root of unity. If  $t \in [1, p-1]$, then,
	$$\displaystyle \left |  \sum_{n \leq x} \omega^{tn}  \right |\leq \frac{2p }{\pi t}.$$
\end{lem} 

\begin{proof}[\textbf{Proof}] Use the geometric series to compute this simple exponential sum as
	\begin{eqnarray} \label{eq4400SK.340}
		\sum_{n \leq x}\omega^{tn}
		&=& \frac{\omega^{t}-\omega^{t(p-1)}}{1-\omega^{t}} \nonumber.
	\end{eqnarray} 
	Now, observe that $\omega=e^{i2 \pi/p}$, the integers $t \in [1, p-1]$, and $d < p-1$. This data implies that $\pi t/p\ne k \pi $ with $k \in \mathbb{Z}$, so the sine function $\sin(\pi t/p)\ne 0$ is well defined. Using standard manipulations, and $z/2 \leq \sin(z) <z$ for $0<|z|<\pi/2$, the last expression becomes
	\begin{equation}
		\left |\frac{\omega^{t}-\omega^{t(x+1)}}{1-\omega^{t}} \right |\leq 	\left | \frac{2}{\sin( \pi t/ p)} \right | 
		\leq \frac{2p}{\pi t}.
	\end{equation}
\end{proof}
\begin{lem}   \label{lem4400SK.150B}\hypertarget{lem4400SK.150B}   Let \(p\geq 2\), let $x<p-1$ a and let $\omega=e^{i2 \pi/p} $ be a $p$th root of unity. Then,
	$$	\left | 	\sum_{\substack{n\leq x\\\gcd(n,p-1)=1}} \omega^{tn}  \right |\ll \frac{2p ^{1+\delta}\log p}{\pi t} , $$ 
	where $\delta>0$ is a small real number and $t \in [1, p-1]$. 
\end{lem} 

\begin{proof}[\textbf{Proof}] Set $\omega=e^{i2 \pi/p}$. Use the inclusion exclusion principle to rewrite the exponential sum as
	\begin{eqnarray} \label{eq4400SK.340d}
		\sum_{\substack{n\leq x\\\gcd(n,p-1)=1}} \omega^{tn}&=& \sum_{n \leq x} \omega^{tn}  \sum_{\substack{d \mid p-1 \\ d \mid n}}\mu(d)   \\
		&=& \sum_{d \mid p-1} \mu(d) \sum_{\substack{n \leq x \\ d \mid n}} \omega^{tn}\nonumber \\
		& =&\sum_{d \mid p-1} \mu(d) \sum_{m \leq (p-1)/ d} \omega^{dtm} \nonumber\\
		&=& \sum_{d \mid p-1} \mu(d) \left( \frac{1-\omega^{dt(\frac{p-1}{d}+1)}}{1-\omega^{dt}}-1\right)  \nonumber\\
		&=& \sum_{d \mid p-1} \mu(d)  \frac{\omega^{dt}-\omega^{dt(\frac{p-1}{d}+1)}}{1-\omega^{dt}} \nonumber.
	\end{eqnarray} 
	Now, observe that the integers $t \in [1, p-1]$, and $d \leq p-1$. This data implies that $\pi dt/p\ne k \pi $ with $k \in \mathbb{Z}$, so the sine function $\sin(\pi dt/p)\ne 0$ is well defined. Using standard manipulations, and $z/2 \leq \sin(z) <z$ for $0<|z|<\pi/2$, the last expression becomes
	\begin{equation}\label{eq4400SK.340h}
		\left |\frac{\omega^{dt}-\omega^{dt(\frac{p-1}{d}+1)}}{1-\omega^{dt}} \right |\leq 	\left | \frac{2}{\sin( \pi dt/p)} \right | 
		\leq \frac{2p}{\pi t},
	\end{equation}
	where $ d=1$ maximizes the inequality. Finally, the upper bound is
	\begin{eqnarray}\label{eq4400SK.340j}
		\left|  \sum_{d \mid p-1} \mu(d) \frac{\omega^{dt}-\omega^{dtp}}{1-\omega^{dt}} \right| 
		&\leq&\frac{2p}{\pi t} \sum_{d \mid p-1} 1 \\
		&\ll& \frac{2p ^{1+\delta}\log p}{\pi t} \nonumber ,
	\end{eqnarray}
	since $\sum_{d \mid n} 1\ll n^{\delta}$, $\delta>0$.
\end{proof}
Additional information on the order of the divisor function $\sum_{d\mid n}1=n^{(\log 2+o(1)/\log\log n)}$ appears in \cite{AT1976}, {\color{red}\cite[Proposition 7.12]{DL2012}}, \cite{HW1979} and similar sources.

\subsection{Harmonic Summation Kernels}\label{S5534}
The harmonic summation kernels naturally arise in the partial sums of Fourier series and in the studies of convergences of continuous functions.

\begin{dfn} \label{dfn5534.100}\hypertarget{dfn5534.100} The Dirichlet kernel is defined by
	\begin{equation} \label{eq5534.200c}
		\mathcal{D}_x(z)=\sum_{-x\leq n \leq x} e^{i 2nz}=\frac{\sin((2x+1)z)}{ \sin \left ( z \right )},
	\end{equation} 
	where $x\in \N$ is an integer and $ z \in \R-\pi\Z$ is a real number. 
\end{dfn}

\begin{dfn} \label{dfn5534.102} \hypertarget{dfn5534.102}The Fejer kernel is defined by
	\begin{equation} \label{eq5534.204}
		\mathcal{F}_x(z)=\sum_{0\leq n \leq x,} \sum_{-n\leq k \leq n} e^{i 2kz}=\frac{1}{2}\frac{\sin((x+1)z)^2}{ \sin \left ( z \right )^2},
	\end{equation} 
	where $x\in \N$ is an integer and $ z \in \R-\pi\Z$ is a real number. 
\end{dfn}

These formulas are well known, see \cite{KT1989} and similar references. For $z \ne k \pi$, the harmonic summation kernels have the upper bounds $\left |\mathcal{K}_x(z) \right |=\left |\mathcal{D}_x(z) \right | \ll |x|$, and $\left |\mathcal{K}_x(z) \right |=\left |\mathcal{F}_x(z) \right | \ll |x^2|$. \\

An important property is the that a proper choice of the parameter $x\geq1$ can shifts the sporadic large value of the reciprocal sine function $1/\sin z$ to $\mathcal{K}_x(z)$, and the term $1/\sin(2x+1)z$ remains bounded. This principle will be applied to the lacunary sequence $\{p_n : n \geq 1\}$, which maximize the reciprocal sine function $1/\sin z$, to obtain an effective upper bound of the function $1/\sin z$.\\

The Dirichlet kernel in \hyperlink{dfn5534.100}{Definition} \ref{dfn5534.100} is a well defined continued function of two variables $x,z \in \R$. Hence, for fixed $z$, it has an analytic continuation to all the real numbers $x \in \R$. \\

A measure theoretical result for the maximal magnitude of the Dirichlet kernel, Fejer and other related exponential sums is stated here. 
\begin{thm}  {\normalfont {\color{red}(\cite[Theorem 1]{FY2013})}} \label{thm4400SK.150B}\hypertarget{thm4400SK.150B} Let $N\geq 2$ be a large integer and let $\mathscr{U}\subset[-N,N]$ be a large subset of cardinality $\#\mathscr{U}=N+1$.  If $\T=\{s\in\C:|s|=1\}$, then the subset of real numbers $\mathscr{A}=\{\alpha\in \T\}$ such that the exponential sum
$$\sum_{u\in \mathscr{U}}e^{i2\pi \alpha u}\geq \frac{2\sqrt{2}}{\pi}(N+1)$$	
has a large magnitude has cardinality
$$\#\mathscr{A}=\frac{2\theta}{N+1}+O\left(\frac{1}{N^3} \right),$$ 
where $\theta \in (0,1)$. Furthermore, the largest measure $\mathscr{A}$ is attained if and only if $\mathscr{U}$ is an arithmetic progression.
\end{thm}

This is related to the cosine problem, which asks for the maximal of the finite sum
\begin{equation} \label{eq5534.200a}
\sum_{u\in \mathscr{U}}\cos \alpha u,
\end{equation} 
and the Littlewood conjecture, which states that
	\begin{equation} \label{eq5534.200b}
\sum_{u\in \mathscr{U}}e^{i2\pi \alpha u}\gg\log N.
\end{equation} 
The lower bound
	\begin{equation} \label{eq5534.200f}
	\sum_{u\in \mathscr{U}}e^{i2\pi \alpha u}\gg\frac{\log N}{(\log\log N)^2},
\end{equation} 
which very close to the claim is achieved in \cite{PS1980}. Observe that \eqref{eq5534.200f} leads to the norm lower bound
\begin{eqnarray} \label{eq5534.200d}
\sum_{1\leq u\leq N}e^{i2\pi \alpha u}\sum_{1\leq v\leq N}e^{-i2\pi \alpha v} &=&\sum_{1\leq v\leq N}\sum_{1\leq u\leq N}e^{i2\pi \alpha (u-v)}\\
&=&\sum_{\substack{1\leq u,v\leq N\\u=v}}e^{i2\pi \alpha (u-v)}+\sum_{\substack{1\leq u,v\leq N\\u\ne v}}e^{i2\pi \alpha (u-v)}\nonumber\\
&\gg&N+N\frac{\log N}{(\log\log N)^2}\nonumber.
\end{eqnarray} 

\section{\textbf{Estimates Of Exponential Sums}} \label{S9933Q}\hypertarget{S9933Q}
Exponential sums indexed by the powers of elements of nontrivial orders have applications in mathematics and cryptography. These applications have propelled the development of these exponential sums. 

\begin{thm}  {\normalfont {\color{red}(\cite[Lemma 4]{FS2002})} }  \label{thm9933Q.110A}\hypertarget{thm9933Q.110A} For integers \(a, k, N\in \mathbb{N}\), assume that \(\gcd (a,N)=c\), and that \(\gcd (k,t)=d\).
	\begin{enumerate} [font=\normalfont, label=(\roman*)]
		\item If the element \(\theta \in \mathbb{Z}_N\) is of multiplicative order \(t\geq t_0\), then 
		\begin{equation}
			\max_{1\leq a\leq p-1} \left| \sum _{ 1\leq x\leq t} e^{i2\pi a \theta ^{k x}/N} \right| <c d^{1/2}N^{1/2} .
		\end{equation} 
		\item If \(H\subset \mathbb{Z}/N \mathbb{Z}\) is a subset of cardinality \(\# H\geq N^{\delta }, \delta >0\), then 
		\begin{equation}
			\max_{\gcd (a,\varphi(N))=1}\left|  \sum _{ x\in H} e^{i2\pi a\theta ^x/N} \right| <N^{1-\delta } .
		\end{equation}
	\end{enumerate}
\end{thm}
Various upper bounds of exponential sums over subsets of elements in finite rings $\left (\mathbb{Z}/N\mathbb{Z}\right )^\times$ can be used to prove the next result, which is used in the proof of \hyperlink{S5757SFT.200}{Lemma} \ref{lem5757SFT.200} . The reader should consult the literature, such as \cite{CC2009}, \cite{BN2004}, {\color{red}\cite[Theorem 2.1]{BJ2007}}, \cite{KS2012}, and references within the cited papers. \\

\subsection{Incomplete Exponential Sums over Consecutive Index}
The simple finite Fourier transform identity 
\begin{equation}\label{eq9933Q.110c}
	\sum_{ n \leq x}  f(n)= \sum_{ n \leq x}  \mathcal{K}(f(n)),
\end{equation}
see \hyperlink{dfn4400FFT.100}{Definition} \ref{dfn4400FFT.100}, is very effective in computing upper bounds of some exponential sums. An improved version of \hyperlink{thm9933Q.110A}{Theorem} \ref{thm9933Q.110A} and a few other applications are illustrated here.

\begin{thm}  \label{thm9933Q.110}\hypertarget{thm9933Q.110} {\normalfont (\cite{SR1973},  \cite{ML1972}) }  Let \(p\geq 2\) be a large prime, and let \(\tau \in \mathbb{F}_p\) be an element of large multiplicative order $\ord_p(\tau) \mid p-1$. Then, for any $b \in [1, p-1]$,  and $x\leq p-1$,
	\begin{equation}\label{eq9933Q.110d}
		\sum_{ n \leq x}  e^{i2\pi b \tau^{n}/p} \ll p^{1/2}  \log p.
	\end{equation}
	
\end{thm}

\begin{proof}[\textbf{Proof}]  Let $p$ be a prime, and let $f(n)=e^{i 2 \pi b\tau^{n} /p}$, where $\tau$ is a primitive root modulo $p$. Applying the finite summation kernel in \eqref{eq9933Q.110c} yields
	\begin{equation} \label{eq9933Q.110f}
		\sum_{ n \leq x}  e^{i2\pi b \tau^{n}/p}= \sum_{ n \leq x}\frac{1}{p} \sum_{0 \leq t\leq p-1,}  \sum_{1 \leq s\leq p-1} \omega^{t(n-s)}e^{i2\pi b \tau^{s}/p} .
	\end{equation}
	The term $t=0$ contributes $-x/p$, and rearranging it yield
	\begin{eqnarray} \label{eq9933Q.110h}
		\sum_{ n \leq x}  e^{i2\pi b \tau^{n}/p}
		&=&\frac{1}{p} \sum_{ n \leq x,} \sum_{1 \leq t\leq p-1,}  \sum_{1 \leq s\leq p-1} \omega^{t(n-s)}e^{i2\pi b \tau^{s}/p}-\frac{x}{p} \\
		&=&\frac{1}{p}  \sum_{1 \leq t\leq p-1}  \left (\sum_{1 \leq s\leq p-1} \omega^{-ts}e^{i2\pi b \tau^{s}/p} \right ) \left (\sum_{ n \leq x}\omega^{tn} \right )-\frac{x}{p}\nonumber.
	\end{eqnarray}
	Taking absolute value, and applying \hyperlink{lem440SK.150A}{Lemma} \ref{lem4400SK.150A} and \hyperlink{lem1234A.150A}{Lemma} \ref{lem1234A.150A}, yield
	\begin{eqnarray} \label{eq9933Q.110i}
		\left | \sum_{ n \leq x}  e^{i2\pi b \tau^{n}/p} \right |
		&\leq&\frac{1}{p}  \sum_{1 \leq t\leq p-1} \left | \sum_{0 \leq s\leq p-1} \omega^{-ts}e^{i2\pi b \tau^{s}/p} \right | \cdot  \left | \sum_{ n \leq x}\omega^{tn} \right |+ \frac{x}{p}\nonumber \\
		&\ll&\frac{1}{p}  \sum_{1 \leq t\leq p-1} \left ( 2p^{1/2} \log p \right ) \cdot  \left ( \frac{2p}{\pi t} \right )+\frac{x}{p}\\
		&\ll& p^{1/2} \log^2 p\nonumber .
	\end{eqnarray}
The third line in \eqref{eq9933Q.110i} uses the estimate 
	\begin{equation} \label{eq9933Q.110j} 
		\sum_{1 \leq t\leq p-1}\frac{1}{t}\ll \log p\ll \log p.
	\end{equation} 
\end{proof}
This appears to be the best possible upper bound. The above proof generalizes the sum of resolvents method used in \cite{ML1972}. Here, it is reformulated as a finite Fourier transform method, which is applicable to a wide range of functions. A similar upper bound for composite moduli $p=m$ is also proved, [op. cit., equation (2.29)]. \\

\subsection{Incomplete Exponential Sums over Relatively Prime Index}

\begin{thm}  \label{thm9933ERP.120}\hypertarget{thm9933ERP.120} {\normalfont (\cite{SR1973}) }  Let \(p\geq 2\) be a large prime, and let \(\tau \in \mathbb{F}_p\) be an element of large multiplicative order $p-1=\ord_p(\tau)$. If $x\leq p-1$ and $s \in [1, p-1]$, then 
	\begin{equation}
		\sum_{ n \leq x}  e^{i2\pi s \tau^{n}/p} \ll p^{1/2}  \log p.
	\end{equation}
	
\end{thm}

This appears to be the best possible upper bound. A similar upper bound for composite moduli $p=m$ is also proved, [op. cit., equation (2.29)].  A simpler proof and generalization of this exponential is is provided in \cite{ML1972}.

\begin{lem}   \label{lem9933ERP.120}\hypertarget{lem9933ERP.120}  Let \(p\geq 2\) be a large prime, let $x\leq p$ and let $\tau $ be a primitive root modulo $p$. If $s \in [1, p-1]$, then,
	\begin{equation}
		\sum_{\substack{n\leq x\\\gcd(n,p-1)=1}} e^{i2\pi s \tau^n/p} \ll  p^{1/2+\delta} ,
	\end{equation} 
	where $\delta>0 $ is a small real number. 
\end{lem}

\begin{proof}[\textbf{Proof}] Use the inclusion exclusion principle to rewrite the exponential sum as
	\begin{eqnarray}
		\sum_{\substack{n\leq x\\\gcd(n,p-1)=1}} e^{i2\pi s \tau^n/p} &=& \sum_{n \leq p-1} e^{i2\pi s \tau^n/p}\sum_{\substack{d \mid p-1 \\ d \mid n}}\mu(d)  \nonumber \\
		&=&\sum_{d \mid p-1} \mu(d) \sum_{\substack{n \leq p-1 \\ d \mid n}} e^{i2\pi s \tau^n/p}  \\
		&=&\sum_{d \mid p-1} \mu(d) \sum_{m \leq (p-1)/ d} e^{i2\pi s \tau^{dm}/p} \nonumber.
	\end{eqnarray}

	Taking absolute value, and invoking \hyperlink{thm9933ERP.120}{Theorem} \ref{thm9933ERP.120} yield 
	\begin{eqnarray}
		\left\vert  \sum_{\substack{n\leq x\\\gcd(n,p-1)=1}}e^{i2\pi s \tau^{n}/p} \right\vert 
		& \leq &\sum_{d \mid p-1} \vert \mu(d) \vert \left\vert  \sum_{m \leq (p-1)/d} e^{i2\pi s \tau^{dm}/p} \right\vert \nonumber \\
		&\ll&\sum_{d \mid p-1} \vert \mu(d)  \vert \left (\left (\frac{p-1}{d} \right )^{1/2} \log p \right ) \\
		&\ll& (p-1)^{1/2} \log(p-1) \sum_{d \mid p-1} \frac{\vert \mu(d)  \vert}{d^{1/2}} \nonumber\\
		&\ll&  (p-1)^{1/2} \log(p-1)\cdot p^{\delta}\nonumber\\
		&\ll&  p^{1/2+\delta}\nonumber .
	\end{eqnarray}
	The last inequality follows from
	\begin{equation}
		\sum_{d \mid p-1} \frac{\vert \mu(d)  \vert}{d^{1/2}} \leq   \sum_{d \mid p-1} 1\ll p^{\delta}
	\end{equation}
	for any arbitrary small number $\delta >0$, and any sufficiently large prime $p \geq 2$. 
\end{proof}

A different approach to this result appears in {\color{red}\cite[Theorem 6]{FS2000}}, and related results are given in \cite{CC2009}, \cite{FS2001}, \cite{GM2005}, and {\color{red}\cite[Theorem 1]{GK2005}}. The upper bound given in \hyperlink{thm9933Q.110}{Theorem} \ref{thm9933Q.110} seems to be optimum. A different proof, which has a weaker upper bound, appears in {\color{red}\cite[Theorem 6]{FS2000}}, and related results are given in \cite{CC2009}, \cite{FS2001}, \cite{GK2005}, and {\color{red}\cite[Theorem 1]{GK2005}}.
\subsection{Equivalent Exponential Sums over Consecutive Index} 
\label{S9933ERP}\hypertarget{S9933ERP}
For any fixed $ 0 \ne b \in \mathbb{F}_p$, the map $ \tau^n \longrightarrow b \tau^n$ is one-to-one in $\mathbb{F}_p$. Consequently, the subsets 
\begin{equation} \label{eq9933ERP.220c}
	\{ \tau^n: n\leq x \}\quad \text { and } \quad  \{ b\tau^n: n\leq x \} \subset \mathbb{F}_p
\end{equation} have the same cardinalities. As a direct consequence the exponential sums 
\begin{equation} \label{eq9933ERP.220d}
	\sum_{n\leq x} e^{i2\pi b \tau^n/p} \quad \text{ and } \quad \sum_{n\leq x} e^{i2\pi \tau^n/p},
\end{equation}
have the same upper bound up to an error term up to an error term. The result below expresses the first exponential sum in \eqref{eq9933ERP.220d} as a sum of simpler exponential sum and an error term. 

\begin{lem}   \label{lem9933ERP.220}\hypertarget{lem9933ERP.220}  Let \(p\geq 2\) be a large primes. If $\tau $ be a primitive root modulo $p$, then,
	\begin{equation} \sum_{n\leq x} e^{i2\pi b \tau^n/p} = \sum_{n\leq x} e^{i2\pi  \tau^n/p} + O(p^{1/2} \log^2 p),
	\end{equation} 
	for any $ b \in [1, p-1]$. 	
\end{lem} 
\begin{proof}[\textbf{Proof}]  For $b\ne 1$, the exponential sum has the representation 
	\begin{equation} \label{eq9933ERP.220h}
		\sum_{n\leq x} e^{\frac{i2\pi b \tau^n}{p}} 
		=\frac{1}{p} \sum_{1 \leq t\leq p-1} \left ( \sum_{1 \leq s\leq p-1} \omega^{-ts}e^{\frac{i2\pi b \tau^s}{p}}\right )\left (\sum_{n \leq x}   \omega^{tn} \right ) -\frac{\varphi(p)}{p},
	\end{equation} 
	confer equation \eqref{eq9933Q.110h} for more details. And, for $b=1$, 
	\begin{equation} \label{eq9933ERP.220i}
	\sum_{n\leq x} e^{\frac{i2\pi  \tau^n}{p}} 
	=\frac{1}{p} \sum_{1 \leq t\leq p-1} \left ( \sum_{1 \leq s\leq p-1} \omega^{-ts}e^{\frac{i2\pi  \tau^s}{p}}\right )\left (\sum_{n \leq x}   \omega^{tn} \right ) -\frac{\varphi(p)}{p}.
\end{equation} 
Differencing \eqref{eq9933ERP.220h} and \eqref{eq9933ERP.220i} produces 
\begin{eqnarray} \label{eq9933ERP.220j}
S&=&\sum_{n\leq x}e^{i2\pi b \tau^n/p} -\sum_{n\leq x} e^{i2\pi  \tau^n/p}\\ 
		&=&     \frac{1}{p} \sum_{0 \leq t\leq p-1} \left ( \sum_{1 \leq s\leq p-1} \omega^{-ts}e^{\frac{i2\pi b \tau^s}{p}}-\sum_{1 \leq s\leq p-1} \omega^{-ts}e^{\frac{i2\pi  \tau^s}{p}}\right ) \left (\sum_{n \leq x}   \omega^{tn} \right ) \nonumber.
	\end{eqnarray}
	By \hyperlink{lem4400SK.150A}{Lemma} \ref{lem4400SK.150A}, the summation kernel is bounded by
	\begin{equation} \label{eq9933ERP.220k}
		\left |\sum_{n \leq x}   \omega^{tn}  \right | 
		\leq   \frac{2 p } {\pi t}, 
	\end{equation}
	and by \hyperlink{lem1234A.150A}{Lemma} \ref{lem1234A.150A}, the difference of two Gauss sums is bounded by
	\begin{eqnarray} \label{eq9933ERP.220l}
	G	&=&  \left | \sum_{1 \leq s\leq p-1} \omega^{-ts}e^{\frac{i2\pi b \tau^s}{p}}-\sum_{1 \leq s\leq p-1} \omega^{-ts}e^{\frac{i2\pi  \tau^s}{p}}\right  |   \\
		&\leq & \left |  \sum_{1 \leq s\leq p-1} \chi(s) \psi_b(s)\right |+ \left | \sum_{1 \leq s\leq p-1} \chi(s) \psi_1(s) \right | \nonumber \\ 
		&\leq & 4 p^{1/2} \log p, \nonumber
	\end{eqnarray}
	where  $\chi(s)=e^{i \pi s t/p}$, and $ \psi_b(s)=e^{i2\pi b \tau^s/p}$. Taking absolute value in \eqref{eq9933ERP.220j} and replacing  \eqref{eq9933ERP.220k}, and  \eqref{eq9933ERP.220l}, return
\begin{eqnarray} \label{eq9933ERP.220m}
|S|		&=&\left|  	\sum_{\substack{n\leq p-1\\\gcd(n,p-1)=1}} e^{i2\pi b \tau^n/p} -\sum_{\substack{n\leq p-1\\\gcd(n,p-1)=1}} e^{i2\pi  \tau^n/p} \right|  \nonumber\\
		& \leq &      \frac{1}{p} \sum_{0 \leq t\leq p-1} \left ( 4p^{1/2} \log p \right ) \cdot \left ( \frac{2 p } {t} \right ) \\
		&\leq & 8p^{1/2} (\log p)(\log p)( \log \log p )\nonumber\\
		&\leq & 8p^{1/2} \log^2 p \nonumber.
	\end{eqnarray}
\end{proof}

\subsection{Equivalent Exponential Sums over Relatively Prime Index} 
For any fixed primitive root $\tau $ and $0 \ne b \in \mathbb{F}_p$, the maps $ n \longrightarrow \tau^n$ and $ n \longrightarrow b \tau^n$ are one-to-one in $\mathbb{F}_p$. Consequently, the subsets 
\begin{equation} \label{eq9933RPI.220c}
	\{ \tau^n: n\leq x \}\quad \text { and } \quad  \{ b\tau^n: n\leq x \} \subset \mathbb{F}_p
\end{equation} have the same cardinalities. As a direct consequence the exponential sums 
\begin{equation} \label{eq9933RPI.220d}
	\sum_{\substack{n\leq x\\\gcd(n,p-1)=1}} e^{i2\pi b \tau^n/p} \quad \text{ and } \quad 	\sum_{\substack{n\leq x\\\gcd(n,p-1)=1}} e^{i2\pi \tau^n/p},
\end{equation}
have the same upper bound up to an error term up to an error term. An asymptotic formula is provided in \hyperlink{lem9933RPI.220}{Lemma} \ref{lem9933RPI.220}. The proof is based on finite Fourier transform version of the Lagrange resolvent 
\begin{equation} \label{344}
	(\omega^t,\zeta^s)=\zeta^s+\omega^{-t} \zeta^{s\tau}+\omega^{-2t} \zeta^{s\tau^{2}}+ \cdots +\omega^{-(p-1)t}\zeta^{s\tau^{p-1}}, 
\end{equation}
where $\omega=e^{i 2 \pi/p}$, $\zeta=e^{i 2 \pi/p}$, and $0 \ne s,t \in \mathbb{F}_p$. \\

The result below expresses the first exponential sum in \eqref{eq9933RPI.220d} as a sum of simpler exponential sum and an error term. 

\begin{lem}   \label{lem9933RPI.220}\hypertarget{lem9933RPI.220}  Let \(p\geq 2\) be a large primes. If $\tau $ be a primitive root modulo $p$, then,
	\begin{equation} 	\sum_{\substack{n\leq x\\\gcd(n,p-1)=1}} e^{i2\pi b \tau^n/p} = 	\sum_{\substack{n\leq x\\\gcd(n,p-1)=1}} e^{i2\pi  \tau^n/p} + O(p^{1/2+\delta}),
	\end{equation} 
	for any $ b \in [1, p-1]$ and $\delta>0$ is a small real number. 	
\end{lem} 
\begin{proof}[\textbf{Proof}]  For $b\ne 1$, the exponential sum has the representation 
	\begin{equation} \label{eq9933RPI.220h}
			\sum_{\substack{n\leq x\\\gcd(n,p-1)=1}} e^{\frac{i2\pi b \tau^n}{p}} 
		=\frac{1}{p} \sum_{1 \leq t\leq p-1} \left ( \sum_{1 \leq s\leq p-1} \omega^{-ts}e^{\frac{i2\pi b \tau^s}{p}}\right )\left (	\sum_{\substack{n\leq x\\\gcd(n,p-1)=1}}   \omega^{tn} \right ) -\frac{\varphi(p)}{p},
	\end{equation} 
	confer equation \eqref{eq9933Q.110h} for more details. And, for $b=1$, 
	\begin{equation} \label{eq9933RPI.220i}
			\sum_{\substack{n\leq x\\\gcd(n,p-1)=1}} e^{\frac{i2\pi  \tau^n}{p}} 
		=\frac{1}{p} \sum_{1 \leq t\leq p-1} \left ( \sum_{1 \leq s\leq p-1} \omega^{-ts}e^{\frac{i2\pi  \tau^s}{p}}\right )\left (	\sum_{\substack{n\leq x\\\gcd(n,p-1)=1}}   \omega^{tn} \right ) -\frac{\varphi(p)}{p}.
	\end{equation} 
	Differencing \eqref{eq9933RPI.220h} and \eqref{eq9933RPI.220i} produces 
	\begin{eqnarray} \label{eq9933RPI.220j}
		S&=&	\sum_{\substack{n\leq x\\\gcd(n,p-1)=1}}e^{i2\pi b \tau^n/p} -	\sum_{\substack{n\leq x\\\gcd(n,p-1)=1}} e^{i2\pi  \tau^n/p}\\ 
		&=&     \frac{1}{p} \sum_{0 \leq t\leq p-1} \left ( \sum_{1 \leq s\leq p-1} \omega^{-ts}e^{\frac{i2\pi b \tau^s}{p}}-\sum_{1 \leq s\leq p-1} \omega^{-ts}e^{\frac{i2\pi  \tau^s}{p}}\right ) \left (	\sum_{\substack{n\leq x\\\gcd(n,p-1)=1}}   \omega^{tn} \right ) \nonumber.
	\end{eqnarray}
	By \hyperlink{lem4400SK.150B}{Lemma} \ref{lem4400SK.150B}, the summation kernel is bounded by
	\begin{equation} \label{eq9933RPI.220k}
		\left |	\sum_{\substack{n\leq x\\\gcd(n,p-1)=1}}   \omega^{tn}  \right | 
		\ll   \frac{2p ^{1+\delta}\log p}{\pi t}, 
	\end{equation}
	and by \hyperlink{lem1234A.150A}{Lemma} \ref{lem1234A.150A}, the difference of two Gauss sums is bounded by
	\begin{eqnarray} \label{eq9933RPI.220l}
		G	&=&  \left | \sum_{1 \leq s\leq p-1} \omega^{-ts}e^{\frac{i2\pi b \tau^s}{p}}-\sum_{1 \leq s\leq p-1} \omega^{-ts}e^{\frac{i2\pi  \tau^s}{p}}\right  |   \\
		&\leq & \left |  \sum_{1 \leq s\leq p-1} \chi(s) \psi_b(s)\right |+ \left | \sum_{1 \leq s\leq p-1} \chi(s) \psi_1(s) \right | \nonumber \\ 
		&\leq & 4 p^{1/2} \log p, \nonumber
	\end{eqnarray}
	where  $\chi(s)=e^{i \pi s t/p}$, and $ \psi_b(s)=e^{i2\pi b \tau^s/p}$. Taking absolute value in \eqref{eq9933RPI.220j} and replacing  \eqref{eq9933RPI.220k}, and  \eqref{eq9933RPI.220l}, return
	\begin{eqnarray} \label{eq9933RPI.220m}
		|S|		&=&\left|  	\sum_{\substack{n\leq x\\\gcd(n,p-1)=1}} e^{i2\pi b \tau^n/p} -\sum_{\substack{n\leq x\\\gcd(n,p-1)=1}} e^{i2\pi  \tau^n/p} \right|  \\
		& \ll &      \frac{1}{p} \sum_{0 \leq t\leq p-1} \left ( 4p^{1/2} \log p \right ) \cdot \left ( \frac{2p ^{1+\delta}\log p}{\pi t} \right ) \nonumber\\
		&\ll & 8p^{1/2+\delta} (\log p)(\log p)(  \log p )\nonumber\\
		&\ll & p^{1/2+\delta} \nonumber,
	\end{eqnarray}
	where $(\log p)^3$ is absorbed by the term $p^{\delta}$. 
\end{proof}

\section{\textbf{Gaussian Sums}}\label{lem1234A}\hypertarget{lem1234A}
Some elementary exponential sums estimates are provided in this section. 
\begin{lem}   \label{lem1234A.150A}\hypertarget{lem1234A.150A}  
{\normalfont (Gauss sums)} Let \(p\geq 2\) be a prime, let $\chi(t)=e^{i2 \pi t/p} $ and  $\psi(t)=e^{i2\pi  \tau^t/p}$ be a pair of characters. Then, the Gaussian sum has the upper bound
	\begin{equation} \label{eq3-355}
		\left |\sum_{1 \leq t \leq p-1}    \chi(t) \psi(t) \right | \leq 2 p^{1/2} \log p.\nonumber
	\end{equation}
	
\end{lem}

\begin{lem}   \label{lem1234A.150R}\hypertarget{lem1234A.150R}   Let \(p\geq 2\) be a prime. If $\omega=e^{i 2 \pi/p}$, $\zeta=e^{i 2 \pi/p}$, and $0 \ne s,t \in \mathbb{F}_p$, then,
	the difference of two Lagrange resolvents has the upper bound
	\begin{equation} \label{eq234A.150b}
		\left | (\omega^t,\zeta^{s\tau^{dp}})-(\omega^t,\zeta^{\tau^{dp}}) \right |  \leq 2 p^{1/2} \log p. 
	\end{equation}
\end{lem} 

\begin{proof}[\textbf{Proof}] The proof for $\left | (\omega^t,\zeta^{s\tau^{dp}}) \right |  \leq  p^{1/2} \log p$ appears in \cite{ML1972}. Hence, the difference
	\begin{equation} \label{234A.150d}
		\left | (\omega^t,\zeta^{s\tau^{dp}})-(\omega^t,\zeta^{\tau^{dp}}) \right |  \leq  \left | (\omega^t,\zeta^{s\tau^{dp}}) \right | +\left |(\omega^t,\zeta^{\tau^{dp}}) \right |    \leq 2 p^{1/2} \log p.
	\end{equation}
\end{proof}

\begin{lem}   \label{lem1234A.200Q}\hypertarget{lem1234A.200Q}  
	Let \(p\geq 2\) be a prime and let $\tau\in \F_p$ be a primitive root. If $t\ne0$, then
	\begin{equation} \label{eq1234A.200Qf}
		\sum _{0\leq s<p/2}e^ {\frac{i2\pi\tau ^{2s+1}t}{p}} =\frac{w}{2}\left(\frac{(\tau t)^{-1}}{p} \right)p^{1/2},\nonumber
	\end{equation}
	where $w\ne1$ is a root of unity.
\end{lem} 
\begin{proof}[\textbf{Proof}] Rewrite the finite sum in term of the quadratic symbol $\left( a\;|\;p\right)$ in the form
	\begin{eqnarray} \label{eq1234A.200Qh}
		\sum _{0\leq s<p/2}e^ {\frac{i2\pi\tau ^{2s+1}t}{p}}&=&\frac{1}{2}\sum _{0\leq a<p}\left(1+\left(\frac{a}{p} \right)  \right) e^ {\frac{i2\pi  a\tau t}{p}}\\[.3cm]
		&=&\frac{1}{2}\sum _{0\leq a<p}\left(\frac{a}{p} \right)   e^ {\frac{i2\pi  a\tau t}{p}}\nonumber\\
		&=&\frac{1}{2}\left(\frac{(\tau t)^{-1}}{p} \right)\sum _{0\leq a<p}\left(\frac{a}{p} \right)   e^ {\frac{i2\pi  a}{p}}\nonumber\\
		&=&\frac{w}{2}\left(\frac{(\tau t)^{-1}}{p} \right)p^{1/2}\nonumber,
	\end{eqnarray}
	where $w\in\C$ is a root of unity.
\end{proof}

\section{\textbf{Subset of Squarefree Totients $p-1$}}\label{S1010SFT-S}\hypertarget{S1010SFT-S}
Let $\tP=\{2,3,5,7,\ldots\}$ and $\tP_{sf}=\{p\in\tP:\mu(p-1)\ne0\}$ be the set of prime numbers and the subset of squarefree totients $p-1$ respectively. The corresponding prime counting functions are defined by
\begin{equation}\label{eq1010SFT.400c}
	\pi(x)=\#\{p\leq x: p\in \tP\}
\end{equation}
and
\begin{equation}\label{eq1010SFT.400d1}
	\pi_{sf}(x)=\#\{p\leq x: p\in \tP \text{ and }\mu(p-1)\ne0\},
\end{equation}
respectively. The various expressions involving the logarithm integral can be handled using the asymptotic formula
\begin{equation} \label{eq1010SFT.400d2}
	\li(x)=\int_{2}^{x} \frac{1}{\log z}  d z=\frac{x}{\log x}+O\left (\frac{x}{\log^2 x} \right ).
\end{equation}
\begin{thm} \label{thm1010.400} {\normalfont(\cite{LH1977})} Let $x\geq 1 $ be a large number, and let $B> 1$ be a constant. Then,
	\begin{equation} \label{eq1010SFT.400f}
		\pi_{sf}(x)=a_1^2\li(x) +O\left(\frac{x}{ (\log x)^c  }\right),
	\end{equation}
	where		
	\begin{equation} \label{eq1010SFT.400g}
		a_1=\prod_{p\geq2}\left(1-\frac{1}{ p(p-1)^2  }\right)
	\end{equation}
	is the density constant.
\end{thm}
A more general result for $k$-free numbers $n= p + m$ was derived in \cite{ML1949}. Similar analysis for squarefree numbers also appear in {\color{red}\cite[Theorem 11.22]{MV2007}}, and \cite{PS2003}.\\

The Siegel-Walfisz theorem, {\color{red}\cite[Corollary 5.29.]{IK2004}}, {\color{red}\cite[Corollary 11.19]{MV2007}}, et cetera, and the previous theorem are used here to extends this result to arithmetic progressions $p = qn+1$ with squarefree $p= qn+a$. Toward this ends, define the prime counting function

\begin{equation}\label{eq1010SFT.400d3}
	\pi_{sf}(x,q,a)=\#\{p=qn+a\leq x: p\in \tP \text{ and }\mu(p-1)\ne0\},
\end{equation}
where $1\leq a<q$ and $\gcd(a, q) = 1$.
\begin{thm} \label{thm430.15} Let $x \geq 1$ be a large number, and let $C> 1$ be a constant. Then, 
	\begin{equation} \label{eq1010SFT.400d4}
		\pi_{sf}(x,q,a)=  a_0 \frac{\li(x)} {\varphi(q)}+O\left ( \frac{x}{\log^B x}\right ),
	\end{equation}
	where $\gcd(a,q)=1$ and $q=O(\log^C x)$, with $C\geq B>1$ arbitrary. \\
\end{thm}

\section{\textbf{Evaluation of the Main Term}}\label{S5757SFT-M}\hypertarget{S5757SFT-M}
The main term $M(x)$ is evaluated in this section. The previous results for squarefree totients are used to obtain an asymptotic formula for the average order of normalized totients function over the squarefree totients.

\begin{lem}  \label{lem5757SFT.400M}\hypertarget{lem5757SFT.400M} Let $x \geq 1$ be a large number, and let $B>1$ be a constant. Then, 
	\begin{equation} \label{eq1010SFT.400d5}
		\sum_{p\leq x} \frac{\varphi(p-1)}{p-1}\mu(p-1)^2=a_0^2 \li(x)+O\left ( \frac{x}{\log^B x}\right ),
	\end{equation}
	where $a_0>0$ is the density of squarefree totients. 
\end{lem}

\begin{proof}[\textbf{Proof}] Use the identity $\varphi(n)/n= \sum_{d|n} \mu(d)/d$, see \cite{AT1976}, {\color{red}\cite[p.\ 15]{IK2004}}, and reverse the order of summation to evaluate the finite sum:
	\begin{eqnarray} \label{eq1010SFT.400d6}
		S_0&=&\sum_{p\leq x} \frac{\varphi(p-1)}{p-1}\mu(p-1)^2 \nonumber \\
		&=& \sum_{p\leq x}\mu(p-1)^2 \sum_{d\;|\;p-1} \frac{\mu(d)}{d} \nonumber \\
		&=& \sum_{d\leq x} \frac{\mu(d)}{d} \sum_{ \substack{p\leq x \\d\mid p-1}}\mu(p-1)^2  \\
		&=& \sum_{d\leq x} \frac{\mu(d) }{d}  \cdot \pi_{sf}(x,d,1) \nonumber .
	\end{eqnarray}
	An application of the Mirsky-Siegel-Walfisz theorem \ref{thm430.15} requires a dyadic decomposition
	\begin{eqnarray} \label{eq1010SFT.400d7}
		S_0&=&\sum_{d\leq x} \frac{\mu(d) }{d}  \cdot \pi_{sf}(x,d,1) \\
		&=& \sum_{d\leq \log^C x} \frac{\mu(d) }{d}  \cdot \pi_{sf}(x,d,1) +\sum_{\log^C x \leq d\leq x} \frac{\mu(d) }{d}  \cdot \pi_{sf}(x,d,1)\nonumber\\
		&=& S_1+S_2 \nonumber ,
	\end{eqnarray}
	where $C> 1$ is an arbitrary constant. Use Theorem \ref{thm430.15} to complete the evaluation of the finite sum $S_1$:
	\begin{eqnarray} \label{eq430.56}
		S_1&=&\sum_{d\leq \log^C x} \frac{\mu(d)}{d}  \cdot \pi_{sf}(x,d,1) \nonumber\\
		& = &\sum_{d\leq \log^C x} \frac{\mu(d) }{d} \left ( \frac{a_0\li(x)}{\varphi(d)} +O \left (\frac{x}{\log^B x}\right ) \right )\\
		&=&a_0\li(x) \sum_{d\leq \log^C x} \frac{\mu(d) }{d\varphi(d)} 
		+O \left (\frac{x}{\log^B x} \sum_{d\leq \log^C x} 
		\frac{1}{d} \right ) \nonumber  \\
		&=& a_0  L(1, \mu \varphi) \li(x)+O \left (\frac{x \log \log x}{\log^B x}\right ) \nonumber  .
	\end{eqnarray} 
	Here the canonical density series, confer {\color{red}\cite[p. 16]{MP2004}}, 
	\begin{equation}
		L(s,\mu \varphi)=\sum_{n\geq 1 }  \frac{\mu(n)}{n^{s}\varphi(n)}=\prod_{p \geq 2} \left ( 1-\frac{1}{p^s(p-1)}\right )
	\end{equation}
	converges for $\Re e(s)\geq 1$. Thus, the sum $\sum_{n\leq z}  \mu(n)/n \varphi(n)=a_0+O(\log z/z)$, since $a_0=L(1, \mu \varphi)=\prod_{p \geq 2} \left ( 1-1/p(p-1)\right ).$\\
	
	The estimate for the finite sum $S_2$ uses Brun-Titchmarsh theorem 
\begin{equation}\label{eq1010SFT.400d9}
		\pi(x,q,a) \leq \frac{(2+o(1)x}{\varphi(q) \log(x/q)} \leq \frac{3x}{\varphi(q) \log x},
	\end{equation} 
	for $1\leq q<x$, see \cite{MJ2012}, and {\color{red}\cite[Theorem 6.6]{IK2004}}. Since $\pi_{sf}(x,q,a) \leq \pi(x,q,a)$, the estimated upper bound is
	\begin{eqnarray} \label{eq1010SFT.400d8}
		|S_2|&=&\left | \sum_{\log^C x \leq d \leq x} \frac{\mu(d) }{d}  \cdot \pi_{sf}(x,d,1) \right | \nonumber \\
		& \leq &\frac{3x}{\log x} \sum_{\log^C x \leq d \leq x} \frac{1}{d\varphi(d)} \nonumber\\
		&\ll&\frac{x}{\log^{C+1} x} \sum_{ d \leq x} \frac{1 }{\varphi(d)}\\
		&\ll& \frac{x}{\log^{C} x}  \nonumber.
	\end{eqnarray}     
	The third line uses $1/d \leq 1/\log^Cx$, and fourth line uses $\sum_{n \leq x} 1/\varphi(n) \ll \log x$, see {\color{red}\cite[p.\ 42]{MV2007}}. Combine the expressions $S_0=S_1+S_2$ to obtain
	\begin{eqnarray} \label{eq1010SFT.400l}
		\sum_{d\leq x} \frac{\mu(d) }{d}  \cdot \pi_{sf}(x,d,1)
		&=& S_1+S_2 \nonumber \\
		&=& \left (a_0^2 \li(x) +O\left (\frac{x}{\log^{B} x} \right ) \right ) +O\left (\frac{x}{\log^{C} x} \right )\nonumber\\
		&=& a_0^2 \li(x) +O\left (\frac{x}{\log^{B} x} \right ) \nonumber,
	\end{eqnarray}
	where $C\geq B> 1$ is an arbitrary constant. 
\end{proof}
Related analysis for the set of all primes appears in {\color{red}\cite[Lemma 1]{SP1969}}.


\section{\textbf{Upper Bound of the Error Term}}  \label{S5757SFT-E}\hypertarget{S5757SFT-E}
An estimate of the error term $E(x)$ is computed in this section. 

\begin{lem} \label{lem5757SFT.200} \hypertarget{S5757SFT.200} Let \(p\geq 2\) be a large prime, let \(\psi \neq 1\) be an additive character, and let \(\tau\) be a primitive root mod \(p\). If the element \(u\ne 0\) is not a primitive root, then, 
\begin{equation} \label{eq5757SFT.200d}
		\sum_{x \leq p\leq 2x} \frac{\mu^2(p-1)}{p}\sum_{\substack{n<p\\\gcd(n,p-1)=1}}\sum_{ 1\leq t\leq p-1} e^{i2\pi\frac{(\tau ^n-u)t}{p}}\ll x^{1/2+\delta},
\end{equation} 
where $\delta>0$ is a small number,	for all sufficiently large numbers $1 \leq x\leq p$.
\end{lem}

\begin{proof}[\textbf{Proof}]   By hypothesis $\tau ^n-u\ne 0$. Let $\psi(z)=e^{i 2 \pi z/p}$, and rearrange the triple finite sum in the form
	\begin{eqnarray} \label{eq5757SFT.200h}
		E(x)&=&\sum_{x \leq p\leq 2x} \frac{\mu^2(p-1)}{p}\sum_{\substack{n<p\\\gcd(n,p-1)=1}}\sum_{ 1\leq t\leq p-1} e^{i2\pi\frac{(\tau ^n-u)t}{p}} \\ [.3cm] 
		&= & \sum_{x \leq p\leq 2x} \frac{\mu^2(p-1)}{p} \left (\sum_{ 0< t\leq p-1}  e^{-i 2 \pi \frac{ut}{p}} \right ) \left ( \sum_{\substack{n<p\\\gcd(n,p-1)=1}}e^{i 2 \pi\frac{ \tau ^{n}t}{p}} \right )\nonumber \\[.3cm] 
		&= & \sum_{x \leq p\leq 2x} \frac{\mu^2(p-1)}{p} \left (\sum_{ 0< t\leq p-1}  e^{-i 2 \pi \frac{ut}{p}} \right ) \left ( \sum_{\substack{n<p\\\gcd(n,p-1)=1}}e^{i 2 \pi\frac{ \tau ^{n}t}{p}} + O(p^{1/2+\delta})\right )\nonumber \\[.3cm] 
		&\leq  & \sum_{x \leq p\leq 2x} \frac{1}{p}  \cdot|U_p |\cdot| V_p |\nonumber.
	\end{eqnarray} 
The third line in equation \eqref{eq5757SFT.200h} follows from \hyperlink{lem9933RPI.220}{Lemma} \ref{lem9933RPI.220}, where $b=t$. The first exponential sum $U_p$ has the exact evaluation
	\begin{equation}\label{eq5757SFT.200i}
		| U_p| = \left |\sum_{ 0< t\leq p-1}  e^{-i 2 \pi \frac{ut}{p}} \right |=1,
	\end{equation} 
	where $\sum_{ 0<m\leq p-1} e^{i 2 \pi um/p}=-1$ for any $u \in [1,p-1]$. The second exponential sum $V_p$ has the upper bound
	\begin{eqnarray} \label{eq5757SFT.200j}
		|V_p|&=& \left |\sum_{\substack{n<p\\\gcd(n,p-1)=1}} e^{i2 \pi \tau ^n/p}+ O\left (p^{1/2+\delta} \right ) \right | \\[.3cm] 
		&\ll &\left |\sum_{\substack{n<p\\\gcd(n,p-1)=1}}e^{i2 \pi \tau ^n/p} \right |+p^{1/2+\delta} \nonumber  \\[.3cm] 
		&\ll&  p^{1/2+\delta}  \nonumber,
	\end{eqnarray} 
see \hyperlink{lem9933ERP.120}{Lemma}  \ref{lem9933ERP.120} or  \hyperlink{thm9933Q.110A}{Theorem} \ref{thm9933Q.110A}. Replacing the estimates \eqref{eq5757SFT.200i} and \eqref{eq5757SFT.200j} return
\begin{eqnarray} \label{eq5757SFT.200k}
\sum_{x \leq p\leq 2x} \frac{1}{p}  \cdot\left | U_p \cdot V_p \right | 
&\leq &	\sum_{x \leq p\leq 2x} \frac{1}{p}  \cdot(1) \cdot    p^{1/2+\delta} \\[.3cm] 
&\leq &	\frac{1}{ x^{1/2-\delta}}\sum_{x \leq p\leq 2x} 1    \nonumber \\
&\leq &	\frac{1}{ x^{1/2-\delta}}\cdot x    \nonumber\\[.3cm] 
&\leq &	x^{1/2+\delta}  \nonumber,
\end{eqnarray}
where the number of primes in the short interval $[x,2x]$ has the trivial upper bound $\pi(2x)-\pi(x)\leq x$.
\end{proof}

\section{\textbf{The Main Result for Squarefree Totients}} \label{S5757SFT}\hypertarget{S5757SFT}
Let $\mu(n) =-1, 0, 1$ be the Mobius function. The
weighted characteristic function for primitive roots $u\in\F_p$ in a finite field with squarefree totient $p-1$ satisfies the relation
\begin{equation}\label{eq5757SFT.400d}
	\mu^2(p-1)\Psi(u) =\left \{
	\begin{array}{ll}
		1 &\text{ if } \mu(p-1)\ne 0 \;\text{ and }\; \ord_p(u) = p-1,\\
		0 &\text{ if }\mu(p-1)= 0 \;\text{ and/or } \;\ord_p(u) \ne p-1.
	\end{array} \right .\nonumber
\end{equation}
The definition of the characteristic function $\Psi(u)$ is given in \hyperlink{lem9955.200A}{Lemma} \ref{lem9955.200A}. Given a fixed integer $u \ne \pm 1,v^2$ and the primes counting function 
\begin{equation} \label{eq5757SFT.400i}
	\pi_{sf}(x,u)=\# \{ p \leq x:\mu(p-1)\ne0 \text{ and } \ord_p(u)=p-1 \}.
\end{equation}
The density of the subset of primes with a fixed squarefree primitive root $u\ne \pm 1$ is defined by the limit
\begin{equation} \label{eq5757SFT.400j}
	\delta(u)=\lim_{x \to \infty} \frac{\pi_{sf}(x,u)}{\pi(x)}=\delta(u) .
\end{equation}
\begin{thm} \label{thm5757SFT.100C}Suppose the GRH is true. Then,
\begin{equation} \label{el6.80}
\pi_{sf}(x,u)=\delta(u,q,a) \frac{x}{ \log  x}+O\left(\frac{x \log \log x}{ \log^2  x}\right)\nonumber.
\end{equation}
\end{thm}
\begin{proof}[\textbf{Proof}]Modified the conditional proof of Artin conjecture achieved in \cite{HC1967} to derive this result.
\end{proof}
As explained in {\color{red}\cite[Section 8.1]{MP2004}}, 
the existing primitive roots counting method fails to prove any unconditional result on density of primes and primitive roots. To circumvent this 
obstacle, the proof of the main result below, uses a new primitive roots counting method.

\begin{proof} [{\color{red}Proof of {\normalfont \hyperlink{thm0987SFT.100}{Theorem} \ref{thm0987SFT.100}}}] Let \(x>x_0\) be a large number. Suppose that the fixed integer $u \ne \pm 1,v^2$ is not a primitive root for all primes \(p\geq x_0\), with \(x_0\geq 1\) constant.  Now  consider the sum of the characteristic function over the short interval \([x,2x]\), that is, 
	\begin{equation} \label{eq0987SFT.100c}
		\pi_{sf}(x,u)=\sum_{x \leq p\leq 2x} \mu^2(p-1)\Psi(u)=0.
	\end{equation}
	Replacing the characteristic function, \hyperlink{lem9955.200A}{Lemma} \ref{lem9955.200A}, and expanding the nonexistence equation \eqref{eq0987SFT.100c} yield
	
	\begin{eqnarray} \label{eq0987SFT.100e}
		\pi_{sf}(x,u)&=&\sum_{x \leq p\leq 2x} \mu^2(p-1)\Psi(u) \\ [.3cm]
		&=&\sum_{x \leq p\leq 2x} \mu^2(p-1)\left (\frac{1}{p}\sum_{\substack{n<p\\\gcd(n,p-1)=1}} \sum_{ 0\leq t\leq p-1} e^{i2\pi\frac{(\tau ^n-u)t}{p}}\right) \nonumber\\[.3cm]
		&=&\delta(u)\sum_{x \leq p\leq 2x} \frac{\mu^2(p-1)}{p}\sum_{\substack{n<p\\\gcd(n,p-1)=1}} 1 \nonumber\\[.3cm]
		&& \hskip 1 in +\sum_{x \leq p\leq 2x} \frac{\mu^2(p-1)}{p}\sum_{\substack{n<p\\\gcd(n,p-1)=1}}\sum_{ 1\leq t\leq p-1} e^{i2\pi\frac{(\tau ^n-u)t}{p}}\nonumber\\[.3cm]
		&=&M(x) + E(x)\nonumber,
	\end{eqnarray} 
	where $\delta(u) \geq 0$ is a constant depending on the integers $u$. The main term $M(x)$ is determined by setting $t=0$, and the error term $E(x)$ is determined by $t\ne0$.\\
	
	Take a constant $b=b(c)>1$, depending on $c\geq0$. Applying \hyperlink{lem5757SFT.400M}{Lemma} \ref{lem5757SFT.400M} to the main term, and \hyperlink{S5757SFT.200}{Lemma} \ref{lem5757SFT.200} to the error term yield
	\begin{eqnarray} \label{eq0987SFT.100h}
		\pi_{sf}(x,u)&=&\sum_{x \leq p\leq 2x} \mu^2(p-1)\Psi(u) \\ [.3cm]
		&=&M(x) + E(x) \nonumber\\ [.3cm]
		&=&c(u)\cdot a_1^2 \cdot \left (\li(2x)-\li(x)\right )  +O\left(\frac{x}{(\log x)^c} \right)+O\left({x^{1/2+\varepsilon}} \right) \nonumber\\ [.3cm]
		&=& \delta(u)\cdot \left (\li(2x)-\li(x)\right )+O\left(\frac{x}{(\log x)^c} \right),
	\end{eqnarray} 
	where $\delta(u)=a_{1}^2\cdot c(u)\geq 0$ and $c(u)\geq 0$ is a correction factor depending on $u$ and $c>1$. \\ 
	
But $\delta(u) > 0$ contradicts the hypothesis \eqref{eq0987SFT.100c} for all sufficiently large numbers $x \geq x_0$. Therefore, the short interval $[x,2x]$ contains primes $p$ such that $p-1$ is squarefree and the $u$ is a primitive root mod $p$. 
\end{proof}

In synopsis, the counting function over a short interval $[1,x]$ can be written as
\begin{equation} \label{eq0987SFT.100k}
	\pi_{sf}(x,u)=\sum_{ p\leq x} \mu^2(p-1)\Psi(u)
	=\delta(u)\li(x) +O\left( \frac{x}{(\log x)^c } \right) .
\end{equation} 
This completes the verification. 
The determination of the correction factors $c(u)$ in some primes counting problems are complex problems in algebraic number theory, some cases are discussed in \cite{SP2003}, \cite{LM2011}, \cite{PT2022} and similar references. \\




\begin{thebibliography}{999}
	
	
	
	

\bibitem{AT1976} Apostol, T. M. \textit{\color{red}Introduction to analytic number theory}. Undergraduate Texts in Mathematics. Springer-Verlag, New York-Heidelberg, 1976. \href{https://mathscinet.ams.org/mathscinet-getitem?mr=0434929}{MR0434929}.
	
	
	
	
	
	
	

\bibitem{BE1968}Burgess, D. A., Elliott,  P.D.T.A.\textit{\color{red}The average of the least primitive root.} Mathematika 15 (1968) 39--50. 	\href{https://mathscinet.ams.org/mathscinet-getitem?mr=MR0237454}{MR0237454}.
	
\bibitem{BD1971} Burgess, D. A. \textit{\color{red}The average of the least primitive root modulo $p^2$.} Acta Arith.  18  (1971), 263--271. 	\href{https://mathscinet.ams.org/mathscinet-getitem?mr=MR0291118}{MR0291118}.		
	
	
	
	

\bibitem{BJ2007} Bourgain, J. \textit{\color{red}Exponential sum estimates in finite commutative rings and applications.} J. Anal. Math. 101, (2007),
325--355. \href{https://mathscinet.ams.org/mathscinet-getitem?mr=MR2346549}{MR2346549}.


\bibitem{BN2004} Bourgain, J. \textit{\color{red}New bounds on exponential sums related to the Diffie-Hellman distributions.} C. R. Math. Acad. Sci. Paris, 338, (2004), no. 11, 825--830. \href{https://mathscinet.ams.org/mathscinet-getitem?mr=MR2059656}{MR2059656}.


	
	
	
	

\bibitem{CC2009} Cobeli Cristian. \textit{\color{red}On a Problem of Mordell with Primitive Roots}. \href{http://arxiv.org/abs/0911.2832}{Arxiv.org/abs/0911.2832}.

	
	
	
	
	
	
	
	
\bibitem{DH1937} Davenport, H. \textit{\color{red}On Primitive Roots in Finite Fields.} Q. J. Math., Oxf. Ser. 8, 308--312 (1937).
	\href{https://zbmath.org/0018.10901}{zbMath 0018.10901}.	 
	
	
\bibitem{DLMF} NIST \textit{\color{red}Digital Library of Mathematical Functions.} Olver, F.; ...; McClain, M. eds.; 2024-06-15. \href{https://dlmf.nist.gov/}{NIST, Release 1.2.1.}  

\bibitem{DL2012} De Koninck, J. M.; Luca, F. \textit{\color{red}Analytic number theory. Exploring the anatomy of integers}. Graduate Studies in Mathematics, 134. American Mathematical Society, Providence, RI, 2012.	\href{https://mathscinet.ams.org/mathscinet-getitem?mr=355?2258}{MR35?52258}.

	
	
	
	
	
	
	
	
	
	
\bibitem{ES1957} Erdos, P.; Shapiro, H. \textit{\color{red}On The Least Primitive Root Of A Prime.} Pacific J. Math. 7 (1), 861--865, 1957. 
	\href{https://zbmath.org/0079.06304}{zbMath 0079.06304}.
	
	
	
\bibitem{FS2000} Friedlander, John B.; Hansen, Jan; Shparlinski, Igor E. \textit{\color{red}Character sums with exponential functions.} Mathematika 47 	(2000),  no. 1--2, 75--85 (2002). \href{https://mathscinet.ams.org/mathscinet-getitem?mr=MR1924489}{MR1924489}.
	
\bibitem{FS2001} Friedlander, John B.; Shparlinski, Igor E. \textit{\color{red}Double exponential sums over thin sets}. Proc. Amer. Math. Soc. 129 (2001),  no. 6, 1617--1621.	\href{https://mathscinet.ams.org/mathscinet-getitem?mr=MR1814088}{MR1814088}.
	

\bibitem{FS2002} Friedlander, John B.; Konyagin, Sergei; Shparlinski, Igor E. \textit{\color{red}Some doubly exponential sums over $\Z_m$.} Acta Arith. 105, (2002), no. 4, 349--370. \href{https://mathscinet.ams.org/mathscinet-getitem?mr=MR1932568}{MR1932568}.
	
	

\bibitem{FY2013}Freiman, G.; Yudin, A.\textit{\color{red}On the measure of large values of the modulus of a trigonometric sum.} European Journal of Combinatorics Volume 34, Issue 8, November 2013, Pages 1338--1347. \href{https://mathscinet.ams.org/mathscinet-getitem?mr=MR3082205}{MR3082205}.
	




	

\bibitem{GK2005} Garaev, M. Z.  Karatsuba, A. A. \textit{\color{red}New estimates of double trigonometric sums with exponential functions.} \href{http://arxiv.org/abs/math/0504026}{Arxiv.org/abs/math/0504026}.  \href{https://mathscinet.ams.org/mathscinet-getitem?mr=MR2246404}{MR2246404}.	 

\bibitem{GM2005} Garaev, M. Z. \textit{\color{red}Double exponential sums related to Diffie-Hellman distributions}. Int. Math. Res. Not. 2005,  no. 17, 1005--1014. \href{https://mathscinet.ams.org/mathscinet-getitem?mr=MR2145707}{MR2145707}.	
	
	
\bibitem{GM1984} Gupta, R.; Murty, M. R. \textit{\color{red}A remark on Artin's conjecture.} Invent. Math. 78 (1984), no. 1, 127--130. \href{https://mathscinet.ams.org/mathscinet-getitem?mr=MR0762358}{MR0762358}.	



	
	
	
	
	
	
	
	
	
\bibitem{HB1986} Heath-Brown, D. R. \textit{\color{red}Artins conjecture for primitive roots.} Quart. J. Math. Oxford Ser. (2) 37, No. 145, 27--38, 1986. \href{https://mathscinet.ams.org/mathscinet-getitem?mr=MR0830627}{MR0830627}.
	
\bibitem{HC1967} Hooley, C. \textit{\color{red}On Artins conjecture.} J. Reine Angew. Math. 225, 209--220, 1967. \href{https://mathscinet.ams.org/mathscinet-getitem?mr=MR0207630}{MR0207630}.
	
	
	
\bibitem{HW1979} Hardy, G. H.; Wright, E. M. \textit{\color{red}An introduction to the theory of numbers.} Oxford University Press, Oxford, 2008. \href{https://mathscinet.ams.org/mathscinet-getitem?mr=MR2445243}{MR2445243}
	
	
	
	
	
	
	
	
\bibitem{IK2004} Iwaniec, H.; Kowalski, E. \textit{\color{red}Analytic number theory.} Amer. Math. Soc. Colloquium Publications, 53. American Mathematical Society, Providence, RI, 2004. 	\href{https://mathscinet.ams.org/mathscinet-getitem?mr=2061214}{MR2061214}.	
	
	
	
	
	
	
	
	
	
\bibitem{KS2012} Konyagin, Sergei V.; Shparlinski, Igor E. \textit{\color{red}On the consecutive powers of a primitive root: gaps and exponential sums.} Mathematika  58  (2012),  no. 1, 11--20.	\href{https://mathscinet.ams.org/mathscinet-getitem?mr=MR2891156}{MR2891156}.
	
	
\bibitem{KT1989} Korner, T. W. \textit{\color{red}Fourier analysis.} Second edition. Cambridge University Press, Cambridge, 1989. \href{https://mathscinet.ams.org/mathscinet-getitem?mr=MR4404761}{MR4404761}.
	
	
	
\bibitem{LE1927} Landau, E. \textit{\color{red}Vorlesungen uber Zahlentheorie: Vol.: 2. Aus der analytischen und geometrischen Zahlentheorie.} Chelsea Publishing Co., New York, 1969, [1927]. \href{https://mathscinet.ams.org/mathscinet-getitem?mr=MR0250844}{MR0250844}.
	
\bibitem{LH1977} Lenstra, H. W. \textit{\color{red}On Artin conjecture and Euclid algorithm in global fields.} Invent. Math. 42, (1977), 201--224. \href{https://mathscinet.ams.org/mathscinet-getitem?mr=MR0480413}{MR0480413}.
	
\bibitem{LM2011} Lenstra, H. W.; Stevenhagen, P.; Moree, P. \textit{\color{red}Character sums for primitive root densities.} Math. Proc. Cambridge Philos. Soc. 157 (2014), no. 3, 489--511. \href{http://arxiv.org/abs/1112.4816}{Arxiv.org/abs/1112.4816}.	\href{https://mathscinet.ams.org/mathscinet-getitem?mr=MR3286520}{MR3286520}.
	
	
	
	
	
	

\bibitem{LN1997} Lidl, R.; Niederreiter, H. \textit{\color{red}Finite fields.} Encyclopedia of Mathematics and its Applications, 20. Cambridge University Press, Cambridge, 1997. \href{https://mathscinet.ams.org/mathscinet-getitem?mr=MR1429394}{MR1429394}.
	
	
	
	
	
	
	
	
	
\bibitem{MG1998}Martin, G. \textit{\color{red}Uniform bounds for the least almost-prime primitive root.} Mathematika 45 (1998), no. 1, 191--207. 	\href{https://mathscinet.ams.org/mathscinet-getitem?mr=MR1644370}{MR1644370}.	

\bibitem{MJ2012}  Maynard, J. \textit{\color{red}On the Brun-Titchmarsh theorem.} Acta Arith.  157  (2013),  no. 3, 249--296. \href{http://arxiv.org/abs/1201.1777}{Arxiv.org/abs/1201.1777}. \href{https://mathscinet.ams.org/mathscinet-getitem?mr=MR3019418}{MR3019418}.


\bibitem{ML1972} Mordell, L. J. \textit{\color{red}On the exponential sum $\sum_{1\leq x \leq X} exp (2\pi i(ax+bg^x )/p)$.} Mathematika 19 (1972), 84--87. \href{https://mathscinet.ams.org/mathscinet-getitem?mr=MR0318073}{MR0318073}.	

\bibitem{ML1949} Mirsky, L. \textit{\color{red}The number of representations of an integer as the sum of a prime and a $k$-free integer}. Amer. Math. Monthly 56, (1949). 17--19. \href{https://mathscinet.ams.org/mathscinet-getitem?mr=MR0028335}{MR0028335}.
	
\bibitem{MM1988} Murty, M. Ram. \textit{\color{red}Artin's conjecture for primitive roots.} Math. Intelligencer  10  (1988),  no. 4, 59--67. \href{https://mathscinet.ams.org/mathscinet-getitem?mr=MR0966133}{MR0966133}.
	
	
	
\bibitem{MP2004} Moree, P. \textit{\color{red}Artin's primitive root conjecture -a survey.} Integers 12 (2012), no. 6, 1305--1416. \href{http://arxiv.org/abs/math/0412262}{Arxiv.org/abs/math/0412262}.
	\href{https://mathscinet.ams.org/mathscinet-getitem?mr=MR3011564}{MR3011564}.
	
	
	
	
	


	
	
	
	
	
	
	
	
	
	
\bibitem{MV2007} Montgomery, H. L., Vaughan, R. C. \textit{\color{red}Multiplicative number theory. I}. Classical theory. Cambridge 	University Press, Cambridge, 2007. \href{https://mathscinet.ams.org/mathscinet-getitem?mr=2378655}{MR2378655}.
	
	
	
	
	
	
	
	
	
	
	
	
	
\bibitem{PS1995}  Pappalardi, F.; Shparlinski, I. \textit{\color{red}On Artin's conjecture over function fields.} Finite Fields Appl.  1  (1995),  no. 4, 399--404. \href{https://mathscinet.ams.org/mathscinet-getitem?mr=MR1353988}{MR1353988}.

\bibitem{PS2003} Pappalardi, F.; Saidak, F.; Shparlinski, I. E. \textit{\color{red}Square-free values of the Carmichael function.} J. Number Theory 103 (2003), no. 1, 122--131. \href{https://mathscinet.ams.org/mathscinet-getitem?mr=MR2008070}{MR2008070}.

\bibitem{PS2013} Pappalardi, F.; Susa, A. \textit{\color{red}An analogue of Artin conjecture for multiplicative subgroups of the rationals.} Arch. Math. (Basel) 101, (2013), no. 4, 319-330. \href{https://mathscinet.ams.org/mathscinet-getitem?mr=MR3116653}{MR3116653}.
	
\bibitem{PT2022} Perucca, A.; Sgobba, P.; Tronto, S. 	\textit{\color{red}Kummer theory for number fields via entanglement 		groups.} Manuscripta Math. 169 (2022), no. 1-2, 251--270. 	\href{https://mathscinet.ams.org/mathscinet-getitem?mr=MR4462664}{MR4462664}

\bibitem{PS1980}Pichorides, S. K. \textit{\color{red}On the $L_1$ norm of exponential sums.} Annales de l'institut Fourier, tome 30, no 2 (1980), p. 79--89. \href{https://mathscinet.ams.org/mathscinet-getitem?mr=MR0584272}{MR0584272}.
	
	
	
\bibitem{RH2002} Roskam, H. \textit{\color{red}Artin primitive root conjecture for quadratic fields.} J. Theory Nombres Bordeaux,
	14, (2002), no. 1, 287--324. \href{https://mathscinet.ams.org/mathscinet-getitem?mr=MR1926004}{MR1926004}.
	
\bibitem{RN1999} Rosen, M. \textit{\color{red}Number theory in function fields.} Graduate Texts in Mathematics, 210. Springer-Verlag, New York, 2002. \href{https://mathscinet.ams.org/mathscinet-getitem?mr=MR1876657}{MR1876657}.
	
	
	
	
	
	
	
	
	
	
\bibitem{SP2003} Stevenhagen, P. \textit{\color{red}The correction factor in Artin's primitive root conjecture}. Theor. Nombres Bordeaux 15 (2003), no. 1, 383--391. \href{https://mathscinet.ams.org/mathscinet-getitem?mr=MR2019022}{MR2019022}.	
	
\bibitem{SP1969} Stephens, P. J. \textit{\color{red}An average result for Artin conjecture.} Mathematika 16, (1969), 178--188. \href{https://mathscinet.ams.org/mathscinet-getitem?mr=MR0498449}{MR0498449}.	
	
\bibitem{SR1973} Stoneham, R. G. \textit{\color{red}On the uniform e-distribution of residues within the periods of rational fractions with applications to normal numbers.} Acta Arith.  22  (1973), 371--389. \href{https://mathscinet.ams.org/mathscinet-getitem?mr=MR0318091}{MR0318091}.	
	
	
	
	
	
	
	
	
	
	
	
	
	
	
	
	
	\bibitem{VR1973} Vaughan, R. C. \textit{\color{red}Some applications of Montgomery's sieve.} J. Number Theory 5 (1973), 64--79. \href{https://mathscinet.ams.org/mathscinet-getitem?mr=MR0342476}{MR0342476}.	
	
	
	
	
	\bibitem{WJ1961} Wrench, J. W. \textit{\color{red}Evaluation of Artin's constant and the twin-prime constant.} Math. Comp. 15 (1961) 396--398. \href{https://mathscinet.ams.org/mathscinet-getitem?mr=MR0124305}{MR0124305}.	
	
	\bibitem{WR2001} Winterhof, A. \textit{\color{red}Character sums, primitive elements, and powers in finite fields.} J. Number Theory 91, 2001, no. 1, 153--163. \href{https://mathscinet.ams.org/mathscinet-getitem?mr=MR1869323}{MR1869323}.
	
	
	
\end{thebibliography}
\end{document}